\documentclass{birkjour}

\usepackage{latexsym,amsxtra,amscd,ifthen}
\usepackage{amsfonts}
\usepackage{verbatim}
\usepackage{amsmath}
\usepackage{amsthm}
\usepackage{amssymb}

\numberwithin{equation}{section}

\theoremstyle{plain}
\newtheorem{theorem}{Theorem}%[section]

\newtheorem{prop}[theorem]{Proposition}
\newtheorem{lemma}[theorem]{Lemma}

\newtheorem{conj}[theorem]{Conjecture}

\theoremstyle{definition}

\newtheorem{rema}[theorem]{Remark}
\newcommand{\F}{\ensuremath{\mathbb F}}
\newcommand{\Z}{\ensuremath{\mathbb Z}}

\def\la{\langle}
\def\ra{\rangle}
\def\char{\mbox{char\,}}

\newcommand{\len}{\mathop{\mathrm{len}}\nolimits}
\newcommand{\lcm}{\mathop{\mathrm{lcm}}}
\begin{document}

\title[Homogeneous monoids of linear growth]{Homogeneous finitely presented monoids of linear growth}
%\title{Operads vs Identities:\\
%a dictionary of linear universal algebra}
%\thanks{Partially
%supported by the grant 05-01-01034 of the Russian Basic Research Foundation.}%

\author{Dmitri Piontkovski}

      \address{Department of Mathematics for Economics\\
 National Research University Higher School of Economics\\
Myasnitskaya str. 20\\
Moscow 101990\\
 Russia
% Central Institute of Economics and Mathematics\\
%                      Nakhimovsky prosp. 47, Moscow 117418,  Russia
}
\thanks{The article was prepared within the framework of the Academic Fund Program at the National Research University Higher School of Economics (HSE) in 2017--2018 (grant 17-01-0006) and by the Russian Academic Excellence Project ``5--100''.}

\email{dpiontkovski@hse.ru}



%\date{July 24, 2017}

 \begin{abstract}
 If a finitely generated monoid $M$ is defined by a finite number of degree-preserving relations, then it has linear growth if and only if it can be decomposed into a finite disjoint union of subsets (which we call ``sandwiches'') of the form $a\la w \ra  b$ where $a,b,w \in M$ and $\la w \ra $ denotes the monogenic semigroup generated by $w$. Moreover, the decomposition can be chosen in such a way that the sandwiches are either singletons or ``free'' ones (meaning that all elements $aw^nb$ in each sandwich are pairwise different). So, the minimal number of free sandwiches in such a decompositions becomes a new numerical invariant of a homogeneous (and conjecturally, non-homogeneous) finitely presented monoid of linear growth. 
 \end{abstract}

\maketitle

%\section{A dictionary of universal algebra: \\
%operads vs identities}

%\section{Introduction}

If a semigroup is a disjoint union of a finite number of free monogenic subsemigroups, then it is finitely presented and residually finite~\cite{ar} and has linear growth~\cite{ae}. It is easy to see that the reverse implication does not hold. For example, the monoid with zero $M =\la x,y | xy=0, xx=0 \ra $ is finitely presented with monomial relations (hence, residually finite) and has linear growth. However, $M$ cannot be represented as a finite union of monogenic semigroups since it contains an infinite set $\{ y^nx | n\ge 0\}$ of nilpotent elements. 

Let us call  a monoid $S$ {\em homogeneous} if its relations are degree-preserving with respect to some weight function, that is,  for some set of generators $X$ of $S$ there is a function 
$d:X\to \Z_{>0}$ such that all relations  of $S$ have either the form $w = 0$ (if $S$ contains zero) or $w=u$ with $d(w) = d(u)$, where for a  word $w = x_1 \dots x_k$ (resp., for a word $u$) on the generators we define $d(w)$ to be the sum
$d(x_1) + \dots +d(x_k)$. In particular, any monoid defined by the relations of the form $u=0$ or $u=w$ where the words $u$ and $w$ have the same length  
is homogeneous with $d(x) =1$ for all $x\in X$. 

Given three elements $a,b$ and $w$ of a semigroup $S$, we call the subset $a\la w\ra b = \{ aw^n b | n\ge 0\}$  {\em sandwich}. For example, each singleton $\{ a\}$ is the sandwich $a\la 1 \ra 1 $.  A sandwich 
$a\la w\ra b$ is called {\em free} if its elements $aw^n b$ are pairwise different for all $n\ge 0$. For example, in free monoids all sandwiches containing  two or more  elements are free.

\begin{theorem}
\label{th:1}
Suppose that a monoid $S$ is homogeneous and finitely presented. Then the following conditions are equivalent.

(i) $S$ has at most linear growth;

(ii)$S$ is a finite union of 
sandwiches;

(iii) $S$ is a union of a finite subset and a finite disjoint union of free sandwiches.
\end{theorem}

We refer to the last decomposition as {\em sandwich decomposition}. For example, a sandwich decomposition
of 
 the above monoid $M$ consists of  the finite set $\{0, 1 \}$  and two free sandwiches
$1 \la y \ra x $ and $1 \la y \ra 1 = \la y \ra $.

\begin{proof}
The implication (ii)$\Longrightarrow$(i) is straightforward since in each sandwich $a \la w \ra b$ the number of words $u$ of length $\len u \le n$ %at most $n$ 
is not greater than 
$$
\frac{n-\len(a)-\len(b)}{\len(w)} = O(n). 
$$ 
%Note that we do not use the assumption that $S$ is homogeneous here.
The implication (iii)$\Longrightarrow$(ii) is trivial since any finite set is a finite union of singletons which are trivial sandwiches.

To complete the prove, let us prove  the implications (i)$\Longrightarrow$(ii) and 
(i)\&(ii)$\Longrightarrow$(iii).
Let $A = \F_2 S$ be the semigroup algebra (with common zero, if $S$ contains zero) over the two-element filed. It is $\Z$-graded connected, finitely presented, and has linear growth.
By~\cite[Theorem~3.1]{p17}, it is automaton in the sense of Ufnarovski with respect to any homogeneous finite set of generators. In particular,  $A$ is automaton in the sense of Ufnarovski 
%same sense 
with respect to a minimal set of  generators of $S$. Then the set of normal words in $A$ form a regular language. 
Now, the theorem   follows from a theorem by 
Paun and Salomaa~\cite[Theorem~3.3]{slender} which describes slender regular languages. 

Let us give also another proof which does not use  methods of the theory of finite automata.
  By~\cite[Corollary~2.3]{p17}, it follows that there 
exists  a finite generating set $X$ of $S$ containing the unit of $S$ and a subset $Q \subset X\times X\times X$
such that the set 
$$Y = 
\{ a w^n b | n\ge 0 , (a,b,w) \in Q \} 
$$
form a linear basis of $A$ (moreover, it is the set of normal words of $A$). It follows that either $S = Y \cup\{0\}$ (if $S$ contains zero) or $S = Y$. Since $Y$ is the union of sandwiches $a\la w\ra b$ for $(a,b,w) \in Q$, we get the implication (i)$\Longrightarrow$(ii). 

It remains to show that the set of words $Y$ is a finite disjoint union of  sandwiches (since $Y$ is a subset of the free monoid $\la X \ra$, all these sandwichas are either free or singletons).  
To apply the induction argument, it is sufficient to use the next lemma. 
\end{proof}
  
  \begin{lemma}
  \label{lem:1}
  Suppose that a subset $Z$ of a free monoid is a finite union of sandwiches. Then $Z$
is decomposable into a finite disjoint union of sandwiches. 
% of subsets of the same form  $p \la q \ra r$.
 \end{lemma} 

\begin{proof}
Let  $Z = \bigcup_{i=1}^s U_i $ is a decomposition of $Z$ into a union of $s$ sandwiches.

First let first consider the case $s=2$. Let  
$U_1 = U = a \la w \ra b$ and $U_2 = U' = a'\la w' \ra b'$. We will show that 
the sets $U\cap U'$, $U \cup U'$, $U\setminus U'$ and $U'\setminus U$
are decomposed as the finite disjoint union of sandwiches.

If the intersection $I = U\cap U'$ is finite, then it is a disjoint union of singletons $\{u\} = u\la1\ra 1$. Moreover, in this case the set $U\setminus I$ (respectively, $U'\setminus I$) is a union of a finite number of singletons and the subset $aw^m \la w \ra b$  for some $m\ge 0$ (resp., $a'w'^n \la w' \ra b'$ for some $n\ge 0$).
So, $U \cup U'$ admits the desired decomposition.

Suppose now that $I$ is infinite. Then the two-sided infinite words $w^\infty$ and $w'^\infty$ coincide. It is sufficient to prove our claim for the sets $a w^M \la w \ra b$ and $a'w'^N \la w' \ra b'$ for all sufficiently large $M,N$ in place of $U$ and $U'$ respectively. Then up to a cyclic permutation of letters in $w$ and $w'$ (and possible change of the words $a,a',b,b'$), one can assume that there exist $m,n,p,q$ such that $w^m = w'^n$ and $aw^p = a'w'^q$.  

Now, if $T$ is one of the sets $I$ and $U\setminus I$, then 
$T$ is periodic in the following sense: for large enough $t$ we have $aw^t b \in  T \Longleftrightarrow aw^{t\pm m} b \in T $. Then 
$ T = \{ a w^{tm +t_0} b | m\in \Z_+, t_0\in S \}$ where $S$ is some finite set of nonnegative integers. 
 It follows that $T$ is a  disjoint union of a finite collection of sandwiches of the form $a w^{t} \la w^m \ra b$.
 Analogously, the set $U'\setminus I$ is a finite disjoint union of sandwiches of the form $a' w'^{t} \la w'^n \ra b'$.
 So, the set $U \cup U' = I\cup(U'\setminus I) \cup(U'\setminus I)$ admits the desired decomposition as well.
 
 Now, for $s> 2$ we proceed by the induction.
 If $Z' = \bigcup_{i=1}^{s-1} U_i $ is decomposable into a disjoint union 
 $\bigsqcup_{j=1}^N T_j$ with $T_j = p_j \la q_j \ra r_j$, then
 $$
 Z = Z' \cup U_s = \bigsqcup_{j=1}^N (T_j \cup U_s), 
 $$
 where the sets $T_j \cup U_s$ admit the desired decomposition  by the $s=2$ case.
\end{proof}

\begin{rema}
Note that each finitely generated semigroup of linear growth is a finite union of sandwiches~\cite[Theorem~4.2]{hot} (see also~\cite[Proposition~2.174b]{bbl}). However, for monoids with infinite set of defining relations the conclusion of Theorem~\ref{th:1} may fail (so that the union is not disjoint).

For example, consider a monoid 
$$N = \la a,w,b | ba=0,bw=0,wa=0, a^2=0,b^2=0, aw^{t^2}b=0 \mbox{ for } t\ge 0\ra .
$$  
Then the number $c_n$ of nonzero words of length $n\ge 2$ in $N$ is equal to
3 if $n = 2+t^2$ for some $t\ge 0$ and 4 otherwise (these are the words $w^n, aw^{n-1},w^{n-1}b$, and $aw^{n-2}b$). It follows that $N$ cannot be presented as a disjoint  union of 
subsets of the desired form  since the sequence $\{c_n\}_{n\ge 0}$ is not a sum of a finite number of arithmetic progressions.
\end{rema}

If $S$ is a finitely presented monoid of linear growth (not necessary homogeneous), we do not  know whether there it is a finite disjoint union of free sandwiches and singletons. Ufnarovski~\cite[5.10]{ufn} conjectured that 
each finitely presented algebra of linear growth (in particular, the algebra $\F_2 S$) is automaton. This conjecture fails for  homogeneous algebras over some infinite fields and 
holds for homogeneous algebras over finite fields~\cite{p17}. Note that if the algebra $\F_2 S$ is automaton with respect to some ordering of the monomials on a finite set of generators of $S$, then $S$ is a finite disjoint union of sandwiches and singletons by the same arguments as above. 
%Moreover, the same arguments as above shows that
% one can also assume that each component $a \la w \ra b$ here is either a singleton or all 
% elements $aw^nb$ are pairwise different for all $n\ge 0$. 
So, we can formulate a weaker (in a sense) version of Ufnarovski's conjecture.

\begin{conj}
\label{conj:1}
Each finitely presented monoid $S$ of linear growth is a finite disjoint union of free sandwiches and a finite set.
% subsets of the form  $a \la w \ra b$,
% where either $w=1$ or for each $m> n\ge 0$ we have $a w^m b \ne aw^n b \ne 0$. 
\end{conj} 

Now we can introduce a new invariant for finitely generated monoids. Given such a monoid $S$, let $\gamma(S)$ be the minimal  number $M$ such that $S$ is the disjoint union of  $M$ free sandwiches and a finite set. In particular, for a finite monoid $S$ we have $\gamma(S) = 0$.
If there is no such finite decompositions, we put for $\gamma(S) = \infty$. 
So, Theorem~\ref{th:1} and Conjecture~\ref{conj:1} simply mean that $\gamma(S) < \infty$ if $S$ is a homogeneous (respectively, arbitrary) finitely presented monoid of linear growth.

\begin{prop}
Let  $S$ be a homogeneous monoid 
such that $\gamma(S) = 1$.
Then $S$ is the union of a free monogenic monoid and a finite set. 
\end{prop} 

Note that the above monoid $M$ (which is homogeneous of linear growth) with $\gamma(M) = 2$  cannot be decomposed into a finite union of monogenic semigroups and a finite set (again because $M$
contains an infinite subset $1 \la y \ra x$ of nilpotent elements). 

\begin{proof}
Let $S$ be the disjoint union of a finite set $Y$ and a free sandwich $Z = a\la w \ra b$. 
For $m>>0$, the set $S_m$ of elements  of the degree $m$ in $S$ is either the singleton $\{a w^k b\}$ (if $k = (m-d(a)-d(b))/d(w)$ is integer) or empty. Since the element $w^t$ is nonzero for all $t\ge 0$,  for $m = t d(w)$ with $t>>0$ this set $S_m$ contains $w^t$.
So, $S_{td(w)}$ is non-empty for all $t>>0$, so that $d(a)+d(b) = s d(w) $
for some integer $s$. We conclude that for each $m >>0$ the set $S_m$ is non-empty if and only if $m-sd(w)/d(w)$ is an integer, or   $m = t d(w)$ for some integer $t$. In the last case, we have 
$S_m = \{ w^{t} \}$, so that $S$ is the union of the free monogenic monoid  $\la w \ra$ and a finite set.
\end{proof}

\subsection*{Acknowledgement}

I am grateful to Jan Okninski for fruitful discussions.

\end{document}